\documentclass[12pt]{amsart}
\usepackage{comment}
\usepackage[utf8]{inputenc}
\usepackage[T2A]{fontenc}
\usepackage[russian,english]{babel}
\usepackage{geometry}

\usepackage[dvipsnames]{xcolor}

\usepackage{amsthm,amsmath,amsfonts}
\usepackage{wasysym}
\usepackage{mathtools}
\usepackage{amssymb}
\usepackage{amscd}
\usepackage[shortlabels]{enumitem}
\usepackage{mathrsfs}
\usepackage{tikz}
\usetikzlibrary{calc,arrows,decorations.pathreplacing}
\usepackage{nicefrac, xfrac}
\usepackage{mathtools,xparse}

\usepackage[pagebackref, colorlinks = true,
linkcolor = red,
urlcolor  = blue,
citecolor = blue,
anchorcolor = blue]{hyperref}
\hypersetup{pdftitle={\@title},pdfauthor={\@author}}

\theoremstyle{plain}

\newtheorem{theorem}{Theorem}[section]
\newtheorem{corollary}[theorem]{Corollary}
\newtheorem{lemma}[theorem]{Lemma}

\newtheorem{proposition}[theorem]{Proposition}

\theoremstyle{definition}
\newtheorem{definition}[theorem]{Definition}

\newtheorem*{claim*}{Claim}

\newtheorem*{theorem*}{Theorem}

\renewcommand{\epsilon}{\varepsilon}

\newcommand{\vertiii}[1]{{\left\vert\kern-0.25ex\left\vert\kern-0.25ex\left\vert #1
		\right\vert\kern-0.25ex\right\vert\kern-0.25ex\right\vert}}

\newcommand{\Om}{\ensuremath{\Omega}}

\newcommand{\om}{\ensuremath{\omega}}

\DeclareSymbolFont{bbold}{U}{bbold}{m}{n}
\DeclareSymbolFontAlphabet{\mathbbold}{bbold}

\DeclareMathAlphabet{\mathpzc}{OT1}{pzc}{m}{it}

\DeclareMathOperator{\HD}{HD}
\DeclareMathOperator{\tDHD}{t_{D}HD}

\DeclareMathOperator{\ord}{\textbf{Ord}}
\DeclareMathOperator{\card}{\textbf{Card}}
\DeclareMathOperator{\scard}{card}
\DeclareMathOperator{\Ind}{Ind}
\DeclareMathOperator{\trind}{trind}
\DeclareMathOperator{\trInd}{trInd}

\newcommand{\cB}{\ensuremath{\mathcal{B}}}

\newcommand{\HH}{\ensuremath{\mathbb H}}

\providecommand{\phantomsection}{}
\AtBeginDocument{\let\textlabel\label}
\makeatletter
\newcommand{\mylabel}[2]{\raisebox{.7\normalbaselineskip}{\phantomsection}(#1)%
	\def\@currentlabel{#1}\textlabel{#2}}
\makeatother

\makeatletter
\newcommand\xlabel[2][]{\phantomsection\def\@currentlabelname{#1}\label{#2}}
\makeatother

\ExplSyntaxOn
\NewDocumentCommand{\mathlist}{ O{,} m m }
 {
  \egreg_mathlist:nnn { #1 } { #2 } { #3 }
 }

\seq_new:N \l__egreg_mathlist_seq
\cs_new_protected:Npn \egreg_mathlist:nnn #1 #2 #3
 {
  \seq_set_split:Nnn \l__egreg_mathlist_seq { #1 } { #3 }
  \seq_use:Nnnn \l__egreg_mathlist_seq { #2 } { #2 } { #2 }
 }
\ExplSyntaxOff

\newcommand{\dInd}{\emph{Ind}}

\newcommand{\mInd}{\text{Ind}}

\newcommand{\mIm}{\text{Im}}
\newcommand{\IM}{\emph{Im}}
\newcommand{\dIm}{\text{\IM}}

\newcommand{\mh}{\text{-}}
\newcommand{\mHD}{\text{HD}}

\newcommand{\mcard}{\text{card}}

\allowdisplaybreaks

\numberwithin{equation}{section}
\title[]{The $D$-Variant of Transfinite Hausdorff Dimension}

\date{\today}

\author{Bryce Decker}
\address{}
\email{\href{BryceJDecker@gmail.com}{BryceJDecker@gmail.com} }

\author{Nathan Dalaklis}
\address{Department of Mathematics, University of North Texas, Denton, TX 76203-5118, USA}
\email{\href{NathanDalaklis@my.unt.edu}{NathanDalaklis@my.unt.edu} }

\begin{document}

	\maketitle
	
	\begin{abstract}
        We assign every metric space $X$ the value $\tDHD(X)$, an ordinal number or one of the symbols $-1$ or $\Om$, and we call it the $D$-variant of transfinite Hausdorff dimension of $X$. This ordinal assignment is primarily constructed by way of the $D$-dimension, a transfinite dimension function consistent with the large inductive dimension on finite dimensional metric spaces while also addressing shortcomings of the large transfinite inductive dimension. Similar to Hausdorff dimension, $\tDHD(\cdot)$ is monotone with respect to subspaces, and is a bi-Lipschitz invariant. It is also non-increasing with respect to Lipschitz maps and satisfies a coarse intermediate dimension property. We also show that this new transfinite Hausdorff dimension function addresses the primary goal of transfinite Hausdorff dimension functions; to classify metric spaces with infinite Hausdorff dimension. In particular, we show that if $\tDHD\geq \om_0$, then $\HD(X) = \infty$. $\tDHD(X)<\om_1$ for any separable metric space, and that one can find a metrizable space with $\tDHD(X)$ bounded between a given ordinal and it's successive cardinal with topological dimension $0$.  
	\end{abstract}

\section{Introduction}

The classical Hausdorff dimension introduced by Hausdorff in \cite{hausdorffDimensionUndAusseres1919} is invariant with respect to bi-Lipschitz mappings, monotone and used to detect the structure of the geometry of metrizable spaces that are zero-dimensional with respect to Ind, the large inductive dimension (or equivalently the Lebesgue covering dimension by the Katetov-Morita Theorem, which follows from \cite{kiitiDimensionProductTopological1977} a proof of which can be found in \cite[Theorem 4.1.3]{engelkingDimensionTheory1978} ). Similarly, the transfinite Hausdorff dimension, introduced by Urba\'nski in \cite[Definiition 2.2]{urbanskiTransfiniteHausdorffDimension2009} addresses an analogous structural detection problem where the  Hausdorff dimension of a metrizable space is $+\infty$. The transfinite Hausdorff dimension assigns an element of $\ord$ or one of the symbols $\Om$ or $-1$ to such a metrizable space and is also bi-Lipschitz invariant and monotone. However, the formulation of transfinite Hausdorff dimension is based in the small transfinite inductive dimension, trind.

The goal of this work is to formulate a new bi-Lipschitz invariant transfinite dimension function that is based in $D$-dimension (or equivalently the transfinite kernel dimension, trker, this equivalence was announced by Kozlovskii in \cite{kozlovskiiDveTeoremyMetricheskih1972}). Originally defined by Henderson \cite[Definition $(D(X))$.]{hendersonDimensionNewTransfinite1968}, $D$-dimension is a transfinite dimension function that agrees with the large transfinite inductive dimension, $\trInd$, for finite dimensional metrizable spaces and satisfies several product and sum theorems for which trInd fails. This $D$-variant of transfinite Hausdorff Dimension satisfies many of the properties of Urba\'nski's transfinite Hausdorff dimension function as:

\begin{theorem}\label{BDTHM1}
If $X$ is a metric space and its Hausdorff dimension is finite, then 
\[
D(X) \leq \tDHD(X) \leq  \lfloor\HD(X)\rfloor
\]
Consequently, $\HD(X)=+\infty$ whenever $\tDHD(X)\geq \omega_0$.
\end{theorem}
and 
\begin{theorem}\label{BDTHM2}
If a metric space $X$ has a topological base $\cB_{\alpha}$ with $\scard(\cB_{\alpha})\leq \aleph_\alpha$ and \\$D(X)<\Omega$, then $\tDHD(X) \leq \omega_{\alpha+1}$. So, if $X$ is also separable, then $\tDHD(X)\leq \omega_1$.
\end{theorem} 

Further the $D$-variant of transfinite Hausdorff Dimension detects structure in metrizable spaces for which $D$-dimension fails to do so

\begin{theorem}\label{BDTHM3}
For each $\alpha \in \ord$, there exists a metrizable space, $X_\alpha$, such that $D(X_\alpha)=0$ and $\alpha \leq \tDHD(X_\alpha) \leq \om_{\tau}$ where $\tau\in\ord$ is least such that $\om_{\tau}> \alpha$. 
\end{theorem}
The present paper draws from and improves upon the first authors Ph.D. thesis. 
\subsection{Preliminaries}
Throughout this work, $\ord$ will denote the class of ordinals. An important subclass is the collection of initial ordinals $\card\subset \ord$ defined by the collection of ordinals obtained by the von Neumann cardinal assignment. That is $\scard(A) := \inf\{\alpha\in\ord\;|\; \exists f:\alpha\to A  \text{ a bijection}\}$. We will also use $\aleph_{\alpha}$ notation for cardinals for simplicity where appropriate. Another class of ordinals that appears in this work is that of the additively indecomposable ordinals, $\HH\subset \ord$. Ordinals in $\HH$ are precisely those of the form $\om^{\delta}$ for some $\delta\in\ord$. For any ordinal $\alpha$ we define $\lambda(\alpha)$ to be the largest limit ordinal $\beta$ such that $\beta<\alpha$ and $\lambda(\alpha)=0$ if $\alpha<\beta$ for every limit ordinal $\beta$, and $n(\alpha)$ to be the finite part of $\alpha$. Further background on ordinal analysis can be found in \cite[Chapter 3]{pohlersProofTheoryIntroduction1989} Note that

\begin{lemma}\label{LEM:ords}
For ordinals $\alpha$ and $\beta$, with $\beta\leq \lambda(\alpha)$, we have that 
\begin{equation}\notag
\lambda(\alpha)-\beta=\lambda(\lambda(\alpha)-\beta).
\end{equation}
\begin{proof}
We will show this by transfinite induction on $\beta$. Fix $\alpha$.
If $\beta=0$ then it is obvious.
For the successor case, suppose $\beta=\delta+1$, and that the equality holds for $\delta$. Then 
\begin{align*}
    \lambda(\alpha)-\beta & =\lambda(\alpha)-\delta-1=\lambda(\lambda(\alpha)-\delta)-1 \\ & =\lambda(\lambda(\alpha)-\delta)=\lambda(\lambda(\alpha)-\delta-1+1)\\ & =\lambda(\lambda(\alpha)-\beta+1) =\lambda(\lambda(\alpha)-\beta),
\end{align*} since adding one will be ignored by the limit ordinal function. Lastly for the limit case, suppose the equality holds for all $\delta<\beta$. As $\beta$ is a limit ordinal we have that $\beta=\lambda(\beta)$ and the equality follows.
\end{proof}
\end{lemma} 

In the following section we prove a handful of new properties for $D$-Dimension that are necessary for the development of the $D$-variant of Hausdorff dimension. Here we collect together a handful of known results regarding $D$-dimension and the large inductive dimension for use later on. We begin with the following definition of Ind, additional properties of Ind can be found in \cite[Chapter 2]{engelkingTheoryDimensionsFinite1995a}. 
\begin{definition}\label{def:Ind}
To every normal space $X$ one assigns the \textit{large inductive dimension} of $X$, denoted by $\mInd (X)$, which is an integer larger than or equal to $-1$, or the ``infinite number" $\infty$; the definition of Ind is given by the following conditions:
\begin{enumerate}[label=\normalfont(\roman*)]
    \item $\Ind (X)=-1$ if and only if $X=\emptyset$; \label{Ind:1}
    \item $\Ind (X)\leq n$, where $n=0,1,...,$ if for every closed set $A\subset X$ and each neighbourhood $V\subset X$ which contains the set $A$ there exists an open set $U\subset X$ such that 
    \[
        A \subset U \subset V \text{ and } \Ind (\partial U) \leq n-1 \text{;}
    \]\label{Ind:2}
    \item $\Ind (X)=n$ if $\Ind (X)\leq n$ and $\Ind (X)>n-1$, i.e., the inequality $\Ind (X)\leq n-1$ does not hold;\label{Ind:3}
    \item $\Ind (X)=\infty$ if $\Ind (X)>n$ for all $n\in\{-1\}\cup\om_0$\label{Ind:4}
\end{enumerate}
\end{definition}

Ind satisfies a locally finite sum theorem which we recall here in the context of metrizable spaces.
\begin{theorem}[Locally Finite Sum Theorem for Ind (c.f. \cite{engelkingTheoryDimensionsFinite1995a} 2.3.10)]\label{THM:sumInd}
If a metrizable space X can be represented as the union of a locally finite family $\{F_s\}_{s\in S}$ of closed subspaces such that $\Ind (F_s) \leq n$ for $s\in S$, then $\Ind (X) \leq n$.

\end{theorem}

We also recall the definition of $D$-dimension and some of its properties due to Henderson \cite{hendersonDimensionNewTransfinite1968} additional resources for the study of $D$-dimension include \cite{engelkingDimensionTheory1978, hendersonDimensionNewTransfinite1968, hendersonDimensionIISeparable1968,  kozlovskiiDveTeoremyMetricheskih1972,olszewskiDdimensionMetrizableSpaces1991a}.
\begin{definition}[Defintion $(D(X))$. in \cite{hendersonDimensionNewTransfinite1968}]\label{def:d}
Define $D(\emptyset)=-1$. For a nonempty metrizable space, $X$, the $D$\textit{-dimension of }$X$, $D(X)$, is defined to be the smallest $\beta\in\ord$, if any exist, such that $X=\bigcup \{A_\alpha:0\leq \alpha \leq \lambda(\beta)\}$, where 
\begin{enumerate}[label=\normalfont(\roman*)]
    \item each $A_\alpha$ is a closed finite dimensional $($with respect to \dInd$)$ subset of $X$;\label{D:1}
    \item for each $\delta$, $\bigcup \{A_\alpha:\delta\leq\alpha\leq\lambda(\beta)\}$ is closed in $X$;\label{D:2}
    \item $n(\beta)= \Ind(A_{\lambda(\beta)})$, if $A_{\lambda(\beta)} =\emptyset$ then define $n(\beta)=0$;\label{D:3}
    \item for each $x\in X$, there is a largest $\delta \leq \lambda(\beta)$, such that $x\in A_\delta$. \label{D:4}
\end{enumerate}
If no such $\beta$ exists then we set $D(X)$ equal to the symbol $\Omega$. If the conditions above are satisfied then we will say that $X=\bigcup \{A_\alpha:0\leq \alpha \leq \lambda(\beta)\}$ is a $\beta \mh D$ \textit{representation of} $X$. Whenever $\lambda$ is a limit ordinal we take $\lambda + (-1)=\lambda$. When taking a $D(X)\mh D$ representation of $X$, we use the notation $X=\bigcup\{A_\alpha:0\leq \alpha\leq \gamma\},$ where it is understood that $\gamma=\lambda(D(X))$ and $\mInd(A_\gamma)=n(D(X))$. Thus, we can say this representation is a ($\gamma +\mInd(A_\gamma))\mh D$ representation. 

\end{definition} 

The $D$-Dimension function satisfies a tropical addition theorem as a corollary to its own Locally Finite Sum Theorem.
\begin{theorem}[The Locally Finite Sum Theorem For $D$ (see \cite{hendersonDimensionNewTransfinite1968} Theorem 3)]\label{thrm:dlocalfin}
If $X$ is the union of a locally finite collection of closed subsets each with $D\mh$dimension $\leq \beta$, then $D(X)\leq \beta.$
\end{theorem}
\begin{corollary}[See \cite{hendersonDimensionNewTransfinite1968} Corollary to Theorem 3]\label{cor:dmax}
If $X$ is the union of two closed subsets $A$ and $B$, then $$D(X)=\max\{D(A),D(B)\}.$$
\end{corollary}

Another important property for our discussion of the $D$-variant of transfinite Hausdorff dimension is that the $D$-Dimension can be bounded above by way of considering an excision of a closed set from the space of interest,

\begin{theorem}[See \cite{hendersonDimensionNewTransfinite1968} Theorem 4]\label{thrm:d4}
If $F$ is a closed subset of the space $X$, then 
\[
D(X)\leq \lambda(D(X\setminus F))+\max \{n(D(X\setminus F)), D(F)\}\leq D(X\setminus F)+D(F).
\]
\end{theorem}

\section{Properties $D$-Dimension}

\begin{theorem}\label{thrm:d14}
If $X=X_* \cup X_0$ where $X_*\not = \emptyset$ is closed and $X_0$ is countable then $D(X)=D(X_*).$ 
\begin{proof}
Without loss of generality replace $X_0$ with $X_0\setminus X_*$ so that $X_* \cap X_0=\emptyset$. Further, if $X_0$ is finite it is closed and so $X = X_0\cup X_{*}$ is a representation of $X$ as the union of two closed sets. Thus, Corollary \ref{cor:dmax} applies and the proposition holds. So we may also assume without loss of generality that $X_0$ is countable.

Suppose now that $D(X)<\omega_0$, then $D(X)$ is finite. Since $X_*$ and the singleton sets $\{x\}\subset X_0$ for $x\in X_{0}$ form a locally finite family of closed subspaces, Theorem \ref{THM:sumInd} applies. Taken together with The Subspace Theorem for $\Ind$ (see \cite[Theorem 2.3.6]{engelkingDimensionTheory1978}).
\[
\Ind(X) \leq \max \{\mInd(X_*), 0\}=\mInd(X_*)\leq \Ind(X).
\] 
So, since $D(X) = \Ind(X)$ for spaces of finite $D$-dimension (see \cite[Theorem 1]{hendersonDimensionNewTransfinite1968})
\[
    D(X) = \Ind(X) = \Ind(X_*) = D(X_*)
\]
which completes the argument for the finite $D$-dimensional case.

Assume now that $D(X_*)=\alpha$ for $\alpha \geq \omega_0$, where $\lambda(\alpha)=\gamma$. We will show $D(X)=D(X_*)$. Let $X_*=\bigcup_{\beta \leq \gamma} F_{\beta}$ be an $\alpha\mh D$-representation of $X_*$. Enumerate $X_0$ by taking $a_{(\cdot)}:\om_0\to X_0$, and define $F'_n=F_n\cup \{a_n\}$ for each $n\in \om_0$. Let $F'_{\beta}=F_{\beta}$ for $\beta \geq \omega_0$ . Clearly $X=\bigcup_{\beta \leq \gamma} F'_{\beta}$. We argue this is an $\alpha \mh D$-representation of $X$: 

Recalling Criteria \ref{D:1} of $D$-Dimension, each $F'_{\beta}$ is the union of two closed sets for $\beta < \omega_0$, and $F_{\beta}' = F_{\beta}$ for $\beta \geq \omega_0$. Furthermore from Corollary \ref{cor:dmax}, we have that 
\[
D(F'_n)=\max\{D(F_n),D(\{a_n\})\}=\max\{D(F_n),0\}. \text{ for all } \beta\leq \om_0.
\]
Thus, $\{F'_{\beta}\}_{\beta\leq \gamma}$ is closed and finite dimensional for all $\beta\leq\gamma$ and so Criterion \ref{D:1} of $D$-dimension is satisfied.

By construction, for $0\leq \delta \leq \gamma$, either $\bigcup_{\delta\leq \beta\leq \gamma} F'_{\beta}$ contains no points of $\{a_n\}_{n\in\om_0}$ or all but finitely many.
If the union contains no points of the countable set then by definition of $\{F'_{\beta}\}_{\beta\leq \gamma}$, $\delta \geq \omega_0$ and the union is closed by Criterion \ref{D:2} of $D$-dimension for the representation of $X_*$. 
Suppose now that $\bigcup_{\delta\leq \beta\leq \gamma} F'_{\beta}$ contains all but finitely many points of $X_0$. Then $\delta=m\in \om_0$ and 
\[
\bigcup_{m\leq \beta\leq \gamma} F'_{\beta}=\bigcup_{m\leq \beta\leq \gamma} F_{\beta}\cup \{a_n\}_{n\geq m}.
\]
As $\bigcup_{m\leq \beta\leq \gamma} F_{\beta}$ is closed, it suffices to prove $\bigcup_{m\leq \beta\leq \gamma} F'_{\beta}$ contains the limit points of each subsequence of $\{a_n\}_{n\geq m}$. 

Suppose by way of contradiction that there exists an  $a'\notin \bigcup_{m\leq \beta \leq \gamma} F'_\beta$ such that $a_{n_i}\rightarrow a'$, where $\{a_{n_i}\}_{i\in\om_0}$ is a subsequence of $\{a_n\}_{n\geq m}$. Then $a'\notin \bigcup_{m\leq \beta \leq \gamma} F_\beta$. Now, since $\bigcup_{m\leq \beta \leq \gamma} F_\beta$ is closed in $X_*$, it is closed in $X$. So $U = (\bigcup_{m\leq \beta \leq \gamma} F_\beta)^c$ is open in X. Then all but finitely many of the terms in $\{a_{n_i}\}_{i\in\om_0}$ fall in $U$. Note that 
\[
a'\in U = \left(\bigcup_{m\leq \beta \leq \gamma} F_\beta\right)^c \subseteq \bigcup_{0\leq \beta < m} F_\beta \subseteq \bigcup_{0\leq \beta < m} F'_\beta,
\]
but then for all $i$ large enough such that $n_i\geq m$,  $a_{n_i}$ does not fall in $U$ since $F_{n_i}'$ is the first set in the representation of $X$ that contains $a_{n_i}$. Hence, we have a contradiction. \lightning 

Thus it must be the case that $\bigcup_{\delta\leq \beta\leq \gamma} F'_{\beta}$ is closed for each $\delta\leq \gamma$ and so Criterion \ref{D:2} holds for $\{F'_{\beta}\}_{\beta\leq \gamma}$.
Furthermore, Criteria \ref{D:3} and \ref{D:4} from Definition \ref{def:d} holds for $\{F'_{\beta}\}_{\beta\leq \gamma}$ by construction.
Thus $D(X)=D(X_*)$ for the case of $\alpha \geq \omega_0$. Thus $D(X)\leq D(X_*)$. 

Finally, by the Subspace Theorem for $D$-Dimension \cite[Theorem 2]{hendersonDimensionNewTransfinite1968}, $X\supset X_*$ implies $D(X)\geq D(X_*)$ and so we may conclude that $D(X)=D(X_*)$ for any metrizable space $X$ which can be written as the union of a nonempty closed set and a countable set.
\end{proof}
\end{theorem} 

\begin{theorem}\label{thrm:d18}
If $X$ is a metrizable space and $X=X_0\cup \bigcup_{j\in J} X_j$, where $X_0$ is closed and $\{X_j\}_{j\in J}$ is a collection of closed sets which are locally finite on $X\setminus X_0$, then $$D(X)\leq D(X_0)+D\left(\bigcup_{j\in J} X_j\right ).$$ 

\begin{proof}
Since $X_0$ is closed, Theorem \ref{thrm:d4} applies in conjunction with the Subspace Theorem for $D$-dimension \cite[Theorem 2]{hendersonDimensionNewTransfinite1968}
\[ 
D(X)\leq D(X\setminus X_0)+ D(X_0)\leq D\left(\bigcup_{j\in J} X_j\right)+ D(X_0).
\]
\end{proof}
\end{theorem}

\begin{theorem}\label{thrm:dinter}
If $-1<D(X)<\Omega$, then for all $\alpha \leq\lambda(D(X))$ there exists a closed subset $M\subseteq X$ and $k\in\om_0$ such that $$D(M)=(\lambda(D(X)) -\alpha)+k.$$ 
\end{theorem}
\begin{proof}
First we show that if $\alpha=\lambda(D(X))$ we can find a closed zero dimensional subset. This is easily achieved by taking $M$ to be a finite collection of points from $X$. Then $M$ is a $D$ representation of itself, as it satisfies all criteria listed in Definition \ref{def:d}. This subset would have $D$-dimension zero as $\Ind(M) = 0$. Thus, $k=0$ would suffice.

Now we consider $\alpha<\lambda(D(X))$. Let $\bigcup\{A_\beta : 0\leq \beta\leq \lambda(D(X))\}$ be a $D(X)\mh D$ representation of $X$. For any $\alpha$, define 
 \begin{equation}\label{EQN:arep}
 X_\alpha= \bigcup_{\alpha\leq\beta\leq\lambda(D(X))} A_\beta.
 \end{equation}
 
It follows from Criterion \ref{D:2} of Definition \ref{def:d} that $X_\alpha$ is closed for any $\alpha$. 
By re-indexing each term in the union in \eqref{EQN:arep} so that it begins at $0$ we obtain a $D$ representation of $X_\alpha$ with $\lambda(D(X_\alpha))\leq \lambda(D(X)) -\alpha$.

By way of contradiction suppose that $\lambda(D(X_\alpha))<\lambda(D(X)) -\alpha$. Let 
\begin{equation}\label{EQN:arep'}
X_\alpha=\displaystyle \bigcup_{0\leq\delta\leq\lambda(D(X_\alpha))}A^\prime_\delta
\end{equation}
be a $D(X_\alpha)\mh D$ representation of $X_\alpha$. For each $\alpha\leq \mu\leq\kappa
$ define $$A^{\prime\prime}_\mu \coloneqq A^\prime_{\mu-\alpha},\text{where } \kappa=\alpha+\lambda(D(X_\alpha)).$$
For $\mu<\alpha$, let $A^{\prime\prime}_\mu= A_\mu$. We show the following: 

\begin{claim*}\notag $\displaystyle \bigcup_{\mu\leq\kappa}A^{\prime\prime}_\mu$ is a $(\kappa +\Ind(A^{\prime\prime}_\kappa))\mh D$-representation of $X$.
\end{claim*}
\begin{proof}\renewcommand\qedsymbol{$\blacksquare$}
We have that $\bigcup \{A^{\prime\prime}_\mu : \alpha\leq\mu\leq\kappa\}=X_\alpha$. So, it is clear from \eqref{EQN:arep} that $\displaystyle \bigcup_{\mu\leq\kappa}A^{\prime\prime}_\mu=X$. By construction, Criteria \ref{D:1} and \ref{D:3} of Definition \ref{def:d} are satisfied.
Also, for $\eta<\alpha$ we have 
\begin{align}
\displaystyle \bigcup_{\eta\leq\mu\leq\kappa}A^{\prime\prime}_\mu&=\displaystyle \bigcup_{\eta\leq\mu<\alpha}A^{\prime\prime}_\mu \cup \displaystyle \bigcup_{\alpha\leq\mu\leq\kappa}A^{\prime\prime}_\mu\notag\\&= \bigcup_{\eta\leq\mu<\alpha}A_\mu \cup X_\alpha\label{EQN:closedthing}\\ &= \displaystyle \bigcup_{\eta\leq\mu<\alpha}A_\mu \cup \displaystyle \bigcup_{\alpha\leq\mu\leq\lambda(D(X))}A_\mu\notag\\ &= \bigcup_{\eta\leq\mu\leq\lambda(D(X))}A_\mu. \notag
\end{align}
Thus, by application of criterion \ref{D:2} noting that \eqref{EQN:closedthing} is a tail of the $D(X)$-$D$ representation of $X$, for all $\eta$ the set $\bigcup_{\eta\leq\mu\leq\kappa}A^{\prime\prime}_\mu$ is closed and thus the collection of $A^{\prime\prime}_\mu$ satisfies Criterion \ref{D:2}. Both \eqref{EQN:arep} and \eqref{EQN:arep'} are $D$ representations of $X_{\alpha}$ so for each $x\in X_{\alpha}$ we may choose the maximum of the indices from \eqref{EQN:arep} and \eqref{EQN:arep'} granted by Criterion \ref{D:4} to find the index of the $A^{\prime\prime}_\mu$ collection that satisfies Criteria \ref{D:4}, for those $x\not\in X_{\alpha}$ the choice is immediate by construction.
\end{proof}
But, by assumption,
\[
\kappa=\alpha + \lambda(D(X_\alpha))< \alpha +\lambda(D(X))-\alpha\leq\lambda(D(X)).
\]
And so we have a contradiction to the definition of $D$-dimension as $D(X)$ is the least such $\beta$ and we have found an ordinal $\kappa+\Ind(A''_k)<\beta$ which satisfies the definition of $D$-dimension.\lightning

Then by above we have 
\[
\lambda(D(X))-\alpha = \lambda(D(X_{\alpha})) \leq D(X_\alpha)\leq (\lambda(D(X))-\alpha)+n(D(X)).
\]

Hence we can find some $k$ which makes equality hold.
\end{proof}

\section{The $D$-Variant of Transfinite Hausdorff Dimension}

Let $\mathcal{M}$ be the category of all metric spaces, and $\mathcal{M}_0$ be the category of all separable metric spaces. Let $\mathcal{P}_m(X)$ denote the collection of all metric subspaces of $X$. For $X,Y\in \mathcal{M}$, denote by $\mathcal{L} (X,Y)$ the collection of all Lipschitz maps $f: E\rightarrow Y$, where $E\in \mathcal{P}_m(X)$. Set
\begin{align*} 
\mathcal{L}(X)= \displaystyle \bigcup_{Y\in \mathcal{M}} \mathcal{L}(X,Y)        &   \hspace{2cm}      \mathcal{L}_0(X)= \displaystyle \bigcup_{Y\in \mathcal{M}_0} \mathcal{L}(X,Y).
\end{align*}

Let $\mathcal{L}^s(X)$ and $\mathcal{L}^s_0(X)$ denote the subcollection of surjective maps from their respective sets. We will denote the Lipschitz constant of each map, $f$, by Lip$(f)$.

We now introduce the $D$-variant of transfinite Hausdorff dimension. This definition follows a similar construction as Urba\'nski in \cite{urbanskiTransfiniteHausdorffDimension2009} but uses $D$-dimension (Definition \ref{def:d}) in place of the small transfinite inductive dimension. It can be observed throughout this section that the two transfinite Hausdorff dimension functions have differing properties. It is our hope that in the future a more direct relationship will be established between the two, e.g. one is always less than or equal to the other.

\begin{definition}\label{def:tDHD}
The \textit{$D$-variant of transfinite Hausdorff dimension}, denoted $\tDHD(X)$, of a metric space $X$ is equal to $-1$ if and only if $X=\emptyset$, and is less than or equal to $\alpha\in\ord$ if and only if $D(\mIm(f))\leq\alpha$ for every map $f\in\mathcal{L}(X)$. Then we define our $D$-variant of transfinite Hausdorff dimension of the space $X$ by setting
\[ \tDHD(X) \coloneqq \sup\{D(\mIm (f)): f\in\mathcal{L}(X)\} \geq D(X).\]
If no supremum exists, we set $\tDHD(X)=\Omega$. Thus in either case, $$\tDHD(X)= \sup\{D(\mIm (f)): f\in\mathcal{L}(X)\} \geq D(X).$$
\end{definition}

Directly from the definition we get the following.
\begin{theorem}[The Subspace Theorem]\label{thrm:tdhdsub}
If $X\in \mathcal{M}$ and $E\in\mathcal{P}_m$, then 
\[
\tDHD(E) \leq \tDHD(X).
\]
\end{theorem} 

Since the composition of two Lipschitz maps is again Lipschitz continuous, we get the following.
\begin{theorem}
If $X$ and $Y$ are two metric spaces and $f: X\rightarrow Y$ is a Lipschitz map, then \[
\tDHD(f(X))\leq \tDHD(X).
\]
So, if $f: X\rightarrow Y$ is bi-Lipschitz continuous, then $$\tDHD(Y)= \tDHD(X).$$
\end{theorem} 

As the image of a separable metric space under a continuous map from a metric space to a metric space is a separable metric space, the next theorem is immediate.
\begin{theorem}
If $X\in\mathcal{M}_0$, then
$$\tDHD(X)=\sup\{D(\dIm(f))): f\in\mathcal{L}_0(X)\}\geq D(X).$$
\end{theorem} 

\begin{theorem}\label{thrm:tdhd4}
If $X\in\mathcal{M}$, then $$\tDHD(X)=\sup\{D(\dIm(f))\},$$where the supremum is taken over all maps $f\in\mathcal{L}(X)$, with closed domains or, equivalently, over all $f\in\mathcal{L}^s(X)$, with closed domains.
\begin{proof}
The argument is identical to Urbanski \cite[Theorem 2.6]{urbanskiTransfiniteHausdorffDimension2009} after replacing $\trind$ with $D$. We recall it here for completeness. Suppose $f: M \rightarrow Y$ is a Lipschitz map with $M\in \mathcal{P}_m(X)$. Let $\hat{Y}$ be the metric completion of $Y$. Since $f$ is Lipschitz continuous, it extends uniquely to a Lipschitz continuous map (with the same Lipschitz constant) $\hat{f}: \overline{M}\rightarrow\hat{Y}$. Then the map $\hat{f}: \overline{M}\rightarrow\hat{f}(\overline{M})$ belongs to $\mathcal{L}(X)$, $\overline{M}$ is a closed subspace of $X$, and $D(\hat{f}|_{\overline{M}}(\overline{M}))\geq D(f(M))$. Hence, we are done.
\end{proof}
\end{theorem} 

\begin{lemma}\label{LEM:Ord2}
If $\lambda(\tDHD(X))-\alpha = \om^{\delta}$ then $\lambda(\tDHD(X))-\beta<\omega^\gamma$ for all $\gamma>\delta$ and $\beta<\alpha$.
\end{lemma}
\begin{proof} By way of contradiction assume that $\lambda(\tDHD(X))-\beta\geq \omega^\gamma$. Without loss of generality, assume that there is a $\beta$ such that $\lambda(\tDHD(X))-\beta=\omega^\gamma$ for a $\gamma>\delta$. Then $\lambda(\tDHD(X))=\omega^\gamma+\beta$ and $\lambda(\tDHD(X))=\omega^\delta+\alpha$. Then we have that
    $$\omega^\gamma+\beta=\omega^\delta+\alpha \implies \omega^\gamma=\omega^\delta+(\alpha-\beta).$$ Thus $\omega^\gamma $ is additively decomposible, a contradiction.\lightning
\end{proof}

\begin{lemma}\label{LEM:Ord3}
    Let $\beta<\alpha\in\ord$. If $\lambda(\tDHD(X))-\alpha\in \HH$ and $\lambda(\tDHD(X))-\beta\neq \lambda(\tDHD(X))-\alpha$ then $\lambda(\tDHD(X))-\beta\not\in\HH$ 
\end{lemma}

\begin{proof}
 Recall that in this setting, $\lambda(\tDHD(X))-\alpha=\omega^\delta$ for some $\delta\in\ord$. Consider a $\beta<\alpha$ such that $\lambda(\tDHD(X)) -\beta>\omega^\delta$. By Lemma \ref{LEM:Ord2} and Lemma \ref{LEM:ords} there is a $\xi\in\ord$ such that
\[
\xi<\lambda(\tDHD(X)) -\beta=\lambda(\lambda(\tDHD(X)) -\beta)
\]
where 
\[
\omega^\delta+\xi=\lambda(\lambda(\tDHD(X))-\beta).
\]
Thus,  $\lambda(\tDHD(X))-\beta\not\in\HH$.
\end{proof}

\begin{theorem}\label{THM:CSE}
If $\tDHD(X)<\Omega$, then for all $\alpha\leq \lambda (\tDHD(X))$ there exists a closed set $M$ and $k\in\om_0$ such that $$\tDHD(M)=(\lambda(\tDHD(X))-\alpha)+k.$$
\begin{proof}
Notice, if $\alpha=\lambda(\tDHD(X))$ we may choose $M=\{x\}$, for some $x\in X$, and $k=0$, so that $\tDHD(M)=0$. 
We will prove the rest of the cases by a careful transfinite induction on $\alpha$.
\begin{proof}[Base Case]\renewcommand\qedsymbol{\checkmark}
Notice if $\alpha=0$, then we may choose $M=X$ and the above equality holds.
\end{proof}

In fact, this is true for all finite $\alpha$. Thus, we proceed with the rest of the cases by only considering $\alpha\geq \om_0$. 

\begin{proof}[Successor Case]\renewcommand\qedsymbol{\checkmark} suppose the equality holds for $\alpha -1$ with $k_0\in \mathbb{N}$ and the closed set $M_0$. Take $M_0$ again and let $k=k_0+1$, then $$\tDHD(M_0)=(\lambda(\tDHD(X))-(\alpha-1))+k_0=(\lambda(\tDHD(X))-\alpha +k.$$ So the equality holds for $\alpha$.
\end{proof}

\begin{proof}[Limit Case]\renewcommand\qedsymbol{\checkmark}
We now suppose that for all $\beta < \alpha$ the statement holds. Note that if there is such a $\beta$ which satisfies $\lambda(\tDHD(X)) -\beta = \lambda(\tDHD(X))-\alpha$, we may choose the same subspace for $\alpha$ as the one that satisfies the condition for $\beta$. If this equality does not hold and $\lambda(\tDHD(X))-\alpha\in\HH$, then, by Lemma \ref{LEM:Ord3}, $\lambda(\tDHD(X))-\beta\not\in \HH$, and we continue in the following way which also captures the case where $\lambda(\tDHD(X))-\alpha\not\in \HH$:

Let $\beta<\alpha$ and $M_\beta$ be the corresponding closed subset such that 
\begin{equation}
\tDHD(M_\beta)=(\lambda(\tDHD(X))-\beta) +k_\beta.
\end{equation}
Since $\beta<\alpha$ with $\alpha$ a limit ordinal and $k_{\beta}\in\om_0$ we then have
\begin{equation}
\lambda(\tDHD(X))-\alpha \leq \tDHD(M_\beta).
\end{equation}
Also by Lemma \ref{LEM:ords},
\[
\lambda(\tDHD(M_\beta))>\lambda(\tDHD(X))-\alpha
\]
so, there is an ordinal number $\delta<\lambda(\tDHD(M_\beta))$ (and $\delta<\alpha$) such that 
\[
\lambda(\tDHD(M_\beta))= \lambda(\tDHD(X))-\beta =\lambda(\tDHD(X))-\alpha+\delta.
\]
By induction hypothesis on the space $M_{\beta}$ with $\delta$ there exists a corresponding closed set $F_{\delta}\subseteq M_{\beta}$ with an integer $k_{\delta}$ such that 
\[
    \tDHD(F_{\delta}) = \lambda(\tDHD(M_{\beta}))-\delta+k_{\beta} = \lambda(\tDHD(X))-\alpha+k
\]
Since $M_\beta$ is closed in $X$, $F_{\delta}$ is as well. This concludes the limit case
\end{proof}
Hence the statement holds.
\end{proof}
\end{theorem} 
It should be noted that the limit case in the above proof shows the theorem has limitations, just like for $D$-dimension. However, we do have that if the $\tDHD$ of a metric space is less than $\omega^2$, then the result is true to the intention of the theorem, which is to be able to find subspaces at each limit ordinal less than or equal to the space's dimension.

Recall that $D$-dimension and Ind agree for finite dimensional spaces. The next theorem shows that like Urba\'nski's Transfinite Hausdorff Dimension, the $D$-variant of Transfinite Hausdorff Dimension gives a finer description of spaces of infinite Hausdorff Dimension. 
\begin{theorem}[c.f. Theorem 2.8 in \cite{urbanskiTransfiniteHausdorffDimension2009}]
If $X$ is a metric space and its Hausdorff dimension is finite, then 
\[
D(X) \leq \tDHD(X) \leq  \lfloor\HD(X)\rfloor
\]
Consequently, $\HD(X)=+\infty$ whenever $\tDHD(X)\geq \omega_0$.
\end{theorem}
\begin{proof}
The argument is similar to that in \cite[Theorem 2.8]{urbanskiTransfiniteHausdorffDimension2009}. We have $D(X) \leq$ $\tDHD(X)$ by Definition \ref{def:tDHD}. Since $\mHD(X)<\infty$, it follows from Marczewski's Theorem (see \cite[Result of Theorem 2 (i)]{marczewskiszpilrajnDimensionMesure}) that $D(X)< \infty$ and, applying this theorem once more, we get for every $f\in \mathcal{L}(X)$ that $$\mHD(X)\geq \mHD(\text{Domain}(f)) \geq \mHD(\mIm(f))\geq D(\mIm(f)).$$ So, taking the supremum, we obtain that $\mHD(X)\geq \tDHD(X)$, and, as $\tDHD(X)$ is an integer, we are done.
\end{proof}

The next theorem proves that knowing the weight of the space provides an upperbound on its $D$-variant transfinite Hausdorff dimension.
\begin{theorem}\label{thrm:tdhdbase}
If a metric space $X$ has a topological base of cardinality $\leq \aleph_\alpha$ and \\$D(X)<\Omega$, then $\tDHD(X) \leq \omega_{\alpha+1}$. So, if $X$ is also separable, then $\tDHD(X)\leq \omega_1$.
\end{theorem} 
\begin{proof}
This follows from \cite[Theorem 10]{hendersonDimensionNewTransfinite1968} and \cite[Theorem 4.1.15]{Eng89} after taking a supremum as in Definition \ref{def:tDHD}.
\end{proof}

Just as $D$-dimension, we have the following tropical addition theorem.
\begin{theorem}
If $X$ is a compact metrizable space and $X=A\cup B$, where $A$ and $B$ are closed then $$\tDHD(X)= \max\{\tDHD(A), \tDHD(B)\}.$$
\begin{proof}
Let $M$ be a closed subspace of $X$ and $f:M\rightarrow Y$ where $f\in\mathcal{L}(X)$. Then $f(M\cap A)$ and $f(M\cap B)$ are closed subsets of $f(M)$ and by Corollary \ref{cor:dmax} we have that $$D(f(M))\leq \max\{D(f(M\cap A), D(f(M\cap B))\}\leq \max\{\tDHD(A), \tDHD(B)\}.$$ As $f$ and $M$ were arbitrary, Theorem \ref{thrm:tdhd4} implies that $$\tDHD(X)\leq \max\{\tDHD(A), \tDHD(B)\}.$$ Theorem \ref{thrm:tdhdsub} gives us the reverse inequality.  
\end{proof}
\end{theorem} 

As a direct consequence we have the following.
\begin{theorem}\label{thrm:tdhd10}
If a compact metric space $X$ is a union of finitely many closed subspaces $X_1$, $X_2$,...,$X_n$, then $$\tDHD(X)=\max \{\tDHD(X_1), \tDHD(X_2),...,\tDHD(X_n)\}.$$
\end{theorem} 

\begin{theorem}[The Locally Finite Sum Theorem]
Let $X$ be a compact metric space. If $X$ can be represented as the union of a locally finite family of closed subsets, $\{X_j\}_{j\in J}$, then $$\tDHD(X)=\sup \{\tDHD(X_j):j\in J\}.$$
\begin{proof}
This follows from Theorem \ref{thrm:tdhd10} and as $X$ is compact we may consider a finite sub-family.
\end{proof}
\end{theorem} 

\begin{theorem}\label{thrm:tdhdloc}
If $X$ is a compact metric space and $X=X_0\cup \bigcup_{j\in J} X_j$, where $X_0$ is closed and $\{X_j\}_{j\in J}$ is a collection of closed sets which are locally finite on $X\setminus X_0$, then $$\tDHD(X)\leq \tDHD(X_0)+\tDHD\left(\bigcup_{j\in J} X_j\right).$$
\begin{proof}
Let $M$ be a closed subset of $X$ and $f:M\rightarrow Y$ be Lipschitz continuous. Then $M$ is compact, and as $f$ is continuous, $f(M)$ is also compact. As $f(M)$ is compact, it is closed. We will first show that $\{f(M\cap X_j)\}_{j\in J}$ is locally finite on $f(M)\setminus f(M\cap X_0).$ 

Suppose by way of contradiction that there exists $y\in f(M)\setminus f(M\cup X_0)$ such that the family $\{f(M\cap X_j)\}_{j\in J}$ is not locally finite at $y$. This means there exists an infinite countable subset $\{j_n\}_{n\in \mathbb{N}}$ of $J$, and for each $n\geq 1$ a point $x_n \in M\cap X_{j_n}$ such that $\lim_{n\rightarrow \infty} f(x_n)=y$. Since $M$ is compact, passing to a subsequence, we may assume without loss of generality that $\lim_{n\rightarrow \infty} x_n=x$ for some $x\in M$. But then, the family $\{X_j\}_{j\in J}$ is not locally finite at $x$. Hence $x\in X_0$ and thus $y=f(x)\in f(M\cap X_0)$, a contradiction.\lightning

Thus $\{f(M\cap X_j)\}_{j\in J}$ is locally finite on $f(M)\setminus f(M\cap X_0)$, and since 
\[
f(M)= f(M\cap X_0) \cup f\left(M\cap \bigcup_{j\in J} X_j\right)
\]
Theorem \ref{thrm:d18} applies and we have that
\begin{align*}
D(f(M)) &\leq D(f(M\cap X_0))+ D\left(f\left(M\cap \bigcup_{j\in J} X_j\right)\right)\\
&\leq \tDHD(M\cap X_0) + \tDHD\left(M\cap \bigcup_{j\in J} X_j\right)\\
&\leq \tDHD(X_0) + \tDHD\left(\bigcup_{j\in J} X_j\right)
\end{align*}
Then by Theorem \ref{thrm:tdhd4},
\begin{align*}
\tDHD(X) &\leq \tDHD(X_0) + \tDHD\left(\bigcup_{j\in J} X_j\right).        
\end{align*}
\end{proof}
\end{theorem} 

\begin{theorem}
If $X=X_* \cup X_0$ where $X_*$ is a non-empty closed subset and $X_0$ is a countable subset then $$\tDHD(X)=\tDHD(X_*).$$
\begin{proof}
Let $M$ be a closed subspace of $X$, and $f: M\rightarrow Y$ where $f\in \mathcal{L}^s(X)$. Then $f(M\cap X_*)$ is closed and $f(M\cap X_0)$ is countable. Then by Theorem \ref{thrm:d14} $$D(f(M))=D(f(M\cap X_*)) \leq \tDHD(X_*).$$ As $f$ and $M$ were arbitrary, by taking a supremum, Theorem \ref{thrm:tdhd4} implies that $\tDHD(X)\leq \tDHD(X_*).$ Theorem \ref{thrm:tdhdsub} implies equality holds.
\end{proof}
\end{theorem}

\section{Smirnov's Spaces}
The goal of this section is to apply the $D$-variant of transfinite Hausdorff dimension to topological Smirnov's spaces along with some notable subspaces. Defined first by Smirnov in \cite[Definition 2]{smirnovUniversalSpacesCertain1959}, and later extended by Henderson in \cite[Definition $(Q^{\alpha})$]{hendersonDimensionNewTransfinite1968} to all $\alpha\geq \om_1$ these spaces are a core example for demonstrating properties of transfinite dimension functions. A good account of properties of Smirnov's Spaces can be found in \cite[Example 7.1.33]{engelkingDimensionTheory1978}. Our goal in the first part of this section is to use Smirnov's Spaces as a vehicle for comparison between transfinite Hausdorff dimension and the $D$-variant of transfinite Hausdorff dimension. For this reason we mostly follow the notation and arguments in \cite[Section 5]{urbanskiTransfiniteHausdorffDimension2009}. In this effort, we consider Smirnov's Spaces only for $\alpha<\om_1$ in order to prove similar results for the $D$-variant of transfinite Hausdorff dimension. In the latter part of this section, we will drop the restriction of having $\alpha<\om_1$ and construct important dense subspaces of Smirnov's spaces that show Theorem \ref{BDTHM3}.

With the overall goals of this section in mind, we recall the general definition of a Smirnov space. 

\begin{definition}[$S_\alpha$]\label{def:salpha}
For $0\leq \alpha < \omega_0$, define $S_\alpha$  to be the finite dimensional cell $I^\alpha$. For $\omega_0 \leq \alpha<\omega_1$, if $\alpha=\beta +1$ set $S_\alpha= S_\beta \times I$ and for $\alpha$ a limit ordinal define $S_\alpha= \omega(\bigoplus_{j\in J}X_j)$, the Alexandrov compactification of the disjoint union, and we let $\rho_\alpha$ be any compatible metric bounded by one.. Assume, inductively, that, for each $\omega_1\leq\beta<\alpha$, $S_\beta$ has been defined and that $\rho_\beta$ is a metric function for $S_\beta$ of bound 1. If $\alpha=\beta + 1$, then define $S_\alpha = S_\beta \times I$ and, for $q,q'\in S_\beta$ and $t,t'\in I$, define 
\[
\rho_\alpha((q,t),(q',t'))=\frac{1}{2}[\rho_\beta (q,q')+ |t-t'|].
\]
If $\alpha$ is an uncountable limit ordinal, then define $S_\alpha$ to be the union of a point $p$ and the discrete union of all $S_\beta,$ $0\leq \beta < \alpha.$ Define the metric on each $S_\alpha$ by:
\begin{equation}
    \rho_\alpha(q,q')= 
    \begin{cases} \frac{1}{n(\beta)+1} & \text{if } p=q \text{ and }q'\in S_\beta\text{, } 0\leq \beta < \alpha, \\
      \frac{\rho_\beta(q,q')}{n(\beta)+1} & \text{if } q,q'\in S_\beta, 0\leq \beta < \alpha, \\
      \rho_\alpha(p,q) + \rho_\alpha(p,q') & \text{if } q,q' \text{ belong to different } S_\beta, 0\leq \beta < \alpha. 
   \end{cases}
\end{equation} 
\end{definition}

Smirnov's spaces have the naturally associated $D$-dimension.

\begin{theorem}[Theorem 10 in \cite{hendersonDimensionNewTransfinite1968}]\label{thrm:Salpha}
For each space, $S_\alpha$, we have $D(S_\alpha)=\alpha$.
\end{theorem}

 It will also be important that our spaces have balanced Alexandrov metrics. Let us recall some information about metrics on product spaces relayed in \cite[Section 4]{urbanskiTransfiniteHausdorffDimension2009}. If $(X_1,\rho_1)$ and $(X_2,\rho_2)$ are two arbitrary metric spaces, then $X_1\times X_2$ is a metric space with the metric $\rho$ given by the formula $$\rho((a_1,a_2),(b_1,b_2))=\sqrt{\rho_1^2(a_1,b_1)+\rho_2^2(a_2,b_2)}.$$
Given two sets $A$ and $B$ in a metric space $(X,\rho)$ we define $$\text{dist}_\rho(A,B)\coloneqq \inf \{\rho(a,b) : (a,b)\in A\times B\}$$ and $$\text{Dist}_\rho(A,B)\coloneqq \sup \{\rho(a,b) : (a,b)\in A\times B\}.$$
Let $J$ be a countable infinite set and let $\{(X_j,\rho_j)\}_{j\in J}$ be a collection of compact metric spaces. Let $\omega(\bigoplus_{j\in J} X_j)$ be the topological one point (Alexandrov) compactification of the topological disjoint union $\bigoplus_{j\in J} X_j$. A metric space $(\omega(\bigoplus_{j\in J} X_j),\rho)$ is called a metric one point compactification of $\bigoplus_{j\in J} X_j$ if $\rho$ induces the Alexandrov compactification topology on $\omega(\bigoplus_{j\in J} X_j)$, and for each $j\in J$ the restriction $\rho|_{X_j}$ is proportional to $\rho_j$. The metric $\rho$ is then referred to as an Alexandrov metric. An Alexandrov metric $\rho$ on $\omega(\bigoplus_{j\in J} X_j)$ is called balanced if $$D_\rho \coloneqq \max \left\{\sup_{i,j\in J}\left\{\displaystyle\frac{\text{Dist}_\rho(X_i,X_j)}{\text{dist}_\rho(X_i,X_j)}\right\}, \sup_{j\in J}\left\{\frac{\text{Dist}_\rho(\omega,X_j)}{\text{dist}_\rho(\omega,X_j)}\right\}\right\}<\infty .$$ The number $D_\rho$ is referred to as the balance constant of the metric $\rho$. Then we have the following.

\begin{proposition}[Proposition 4.2 in \cite{urbanskiTransfiniteHausdorffDimension2009}]\label{metric2}
If $J$ is a countable infinite set and if $\{(X_j,\rho_j)\}_{j\in J}$ is a collection of compact metric spaces, then there exists at least one balanced Alexandrov metric on $\omega(\bigoplus_{j\in J} X_j)$.
\end{proposition}

\begin{lemma}[Lemma 4.3 in \cite{urbanskiTransfiniteHausdorffDimension2009}]\label{metric3}
Suppose that $J$ is a countable infinite set, $\{(X_j,\rho_j)\}_{j\in J}$ is a collection of compact metric spaces, and $\rho$ is a balanced Alexandrov metric on $\omega(\bigoplus_{j\in J} X_j)$. Suppose further that for each $j\in J$, $A_j$ is a subset of $X_j$, and $f_j: A_j \rightarrow X_j$ is a Lipschitz continous map with the Lipschitz constant bounded above by the same number $L$. Define the map $$f: \{\omega\}\cup \bigcup_{j\in J}A_j \rightarrow \omega\left(\bigoplus_{j\in J} X_j\right)$$ by requiring that $f(\omega)=\omega$ and $f|_{A_j}=f_j$ for all $j\in J$. Then $f$ is a Lipschitz map with $\text{Lip}(f)\leq \max \{L,D_\rho\}$. Also, $$\dIm(f)=\{\omega\}\cup \bigcup_{j \in J} f_j(A_j).$$
\end{lemma}

We will now recall the construction of a Smirnov's sequence. Let $I$ be the interval $[0,1]$ endowed with its standard Euclidean metric. Define a transfinite sequence $((S_{\alpha}, \rho_{\alpha}))_{\alpha< \omega_1}$ consisting of compact metric spaces in the following way. Set $S_0=\{0\}, S_{\alpha}=S_\beta \times I$ if $\alpha=\beta +1$, and $(S_{\alpha}, \rho_{\alpha})$ as the balanced Alexandrov metric compactification $\omega(\bigoplus_{\beta < \alpha}S_\beta)$ of $\bigoplus_{\beta < \alpha}S_\beta$ if $\alpha$ is a limit ordinal. Proposition \ref{metric2} implies that Smirnov's sequences are well defined.

Now we recall the definition of Smirnov's Cantor sets and sequences.
Suppose $C\subseteq I$ is a topological Cantor set (perfect and totally disconnected) with positive Lebesgue measure $\lambda (C)$. Let $\phi :C \rightarrow I$ be the function given by the formula $$\phi (t)=\lambda (C)^{-1} \lambda([\min(C),t]).$$
Clearly $\phi$ is a Lipschitz continuous map with Lipschitz constant equal to $\lambda(C)^{-1}$ and $\phi(C)=I$. 
\begin{definition}[C$_\alpha$]\label{calpha}
Given a Smirnov's sequence $((S_{\alpha}, \rho_{\alpha}))_{\alpha< \omega_1}$ define $C_0$ to be $\{0\}$, $C_\alpha=C_\beta \times C$ if $\alpha=\beta +1$, and $C_\alpha=\{\omega \} \cup \bigcup_{\beta<\alpha}C_\beta$, if $\alpha < \omega_1$ is a limit ordinal. $((C_{\alpha}, \rho_{\alpha}|_{C_\alpha}))_{\alpha< \omega_1}$ is referred to as a Smirnov's sequence of Cantor sets. If $\alpha \geq \omega_1$ we will define $C_\alpha$ as follows, first assume $C_\beta$ has been defined for all $\beta < \alpha$. If $\alpha=\beta +1$ let $C_\alpha= C_\beta \times C$ and if $\alpha$ is a limit ordinal let $C_\alpha=\{p\}\cup \bigoplus_{\beta < \alpha} C_\beta$. 
\end{definition}

Clearly $C_\alpha \subseteq S_\alpha$ for all $\alpha \in \ord$, and $C_\alpha$ is a topological Cantor set for $\alpha < \omega_1$. Since $\Ind(C)= D(C) = 0$, after application of Theorem \ref{thrm:d4} and Theorem \ref{thrm:dlocalfin} a transfinite induction argument shows that $D(C_\alpha)=0$ for every ordinal $\alpha<\om_1$.

The following two lemmas follow similar arguments to the ones presented by Urba\'nski in \cite{urbanskiTransfiniteHausdorffDimension2009}.
\begin{lemma}\label{lemma:calpha}
If $((S_{\alpha}, \rho_{\alpha}))_{\alpha< \omega_1}$ is a Smirnov's sequence and $((C_{\alpha}, \rho_{\alpha}))_{\alpha< \omega_1}$ is the corresponding Smirnov's sequence of Cantor sets, then $\tDHD(C_\alpha)\geq$ $D(S_\alpha)$ for all $\alpha <\omega_1$.
\begin{proof}
By the same construction in \cite[Lemma 5.2.]{urbanskiTransfiniteHausdorffDimension2009}, one may produce a sequence $(\phi_{\alpha})_{\alpha<\om_1}$ of Lipschitz continuous surjections from $C_{\alpha}$ onto $S_{\alpha}$ with Lipschitz constants bounded above by $\max\{2, \lambda(C)^{-1}\}$. Then, by Definition \ref{def:tDHD} we have for all $\alpha < \omega_1$, $$\tDHD(C_\alpha)\geq D(\mIm(\phi_\alpha))=D(S_\alpha).$$
\end{proof}

\end{lemma} 

\begin{lemma}
 If $((S_{\alpha}, \rho_{\alpha}))_{\alpha< \omega_1}$ is a Smirnov's sequence, then $\tDHD(S_\alpha)<\omega_1$ for all $\alpha <\omega_1$.
 \begin{proof}
  By the same transfinite induction argument in \cite[Lemma 5.3.]{urbanskiTransfiniteHausdorffDimension2009} replacing the transfinite Hausdorff dimension with the $D$-variant of transfinite Hausdorff dimension and replacing \cite[Theorem 3.7]{urbanskiTransfiniteHausdorffDimension2009} with our Theorem \ref{thrm:tdhdloc} we conclude that $$\tDHD(S_\alpha)=\tDHD(S_\gamma \times I^n)<\omega_1.$$
 \end{proof}
\end{lemma}

Then by Theorem \ref{thrm:tdhdsub}, we have that $\tDHD(C_\alpha)\leq \tDHD(S_\alpha)<\omega_1$. By \cite[Theorem 7.3.16, Example 7.2.12]{engelkingDimensionTheory1978} we have that $\sup_{\alpha <\omega_1}\{D(S_\alpha)\}=\omega_1$, applying Lemma \ref{lemma:calpha}, we get the following.

\begin{theorem}[c.f. Theorem 5.4 in \cite{urbanskiTransfiniteHausdorffDimension2009}]
If $((S_{\alpha}, \rho_{\alpha}))_{\alpha< \omega_1}$ is a Smirnov's sequence and $((C_{\alpha}, \rho_{\alpha}))_{\alpha< \omega_1}$ is the corresponding Smirnov's sequence of Cantor sets, then
\begin{enumerate}[label=\normalfont(\roman*)]
    \item $D(S_\alpha)\leq$ $\tDHD(C_\alpha)< \omega_1$
    \item $\sup_{\alpha < w_1} \{\tDHD(C_\alpha)\}=\omega_1$
    \item $\scard(\{\tDHD(C_\alpha)| \alpha < \omega_1 \})= \aleph_1$
    \item The family $((C_{\alpha}, \rho_{\alpha}))_{\alpha< \omega_1}$ contains uncountably many Cantor sets, no two of which are bi-Lipschitz equivalent.
    \item If $\alpha \geq \omega_0$, then $\HD(C_\alpha)=+\infty$
\end{enumerate}
\end{theorem} 

We shall now define dense subsets of each $S_\alpha$.

\begin{definition}[$D_\alpha$]
For $\alpha < \omega_1$, as each space $S_\alpha$ is separable, define $D_\alpha$ to be a countable dense subset of $S_{\alpha}$. For $\omega_1\leq\alpha <\Omega$, construct $D_\alpha$ inductively as follows: First assume that, for $\beta < \alpha$, $D_\beta$ has been defined.
For $\alpha= \beta + 1$, let $D_\alpha=D_\beta \times D_1$.  
For $\alpha$ a limit ordinal, let $D_\alpha= \{p\} \cup \bigoplus_{\beta<\alpha} D_\beta$. 
\end{definition} 

$D_\alpha$'s density in $S_{\alpha}$ is clear in the successor case as it is the product of dense sets. For the limit case, note that $p$ is an accumulation point of $S_\beta$ for some values of $\beta$ and $\bigoplus_{\beta<\alpha} D_\beta$ is dense in $\bigoplus_{\beta<\alpha} S_\beta$.

\begin{lemma}\label{cardD}
If $\alpha < \omega_\tau$, then \emph{Card}$(D_\alpha)\leq \aleph_{\tau -1}$, where $1\leq\tau \in \ord$.
\begin{proof}
We shall prove the above statement by transfinite induction.

\begin{proof}[Base Case]\renewcommand\qedsymbol{\checkmark}
For $\alpha < \omega_1$, $\alpha$ is countable and thus $D_\alpha$ is the countable dense set in $S_\alpha$. Then Card($D_\alpha)\leq \aleph_0$.
\end{proof}

 Note that the only cases remaining are for $\alpha\geq \om_1$, thus without loss of generality we only consider these cases in the inductive steps. Assume the statement holds for all $\beta < \alpha$. Further, again without loss of generality, we will assume $1\leq \tau$ is the least ordinal that satisfies $\alpha < \omega_\tau$.

\begin{proof}[Successor Case]\renewcommand\qedsymbol{\checkmark} $D_\alpha= D_\beta \times D_1$. Then 
\[ 
\mcard(D_\alpha)= \mcard(D_\beta)\times \mcard(D_1) \leq \aleph_{\tau -1} \times \aleph_0= \aleph_{\tau -1}.
\] 
\end{proof}

\begin{proof}[Limit Case]\renewcommand\qedsymbol{\checkmark} $D_\alpha= \{p\} \cup \bigcup_{\beta < \alpha} D_\beta$. Then for each $\beta < \alpha$, we have that 
\[
\scard (D_\beta)\leq \aleph_{\tau -1}.
\]
Thus,  
\begin{align*}
    \mcard(\bigcup_{\beta<\alpha}D_\beta)&= \sup_{\beta < \alpha} \{\mcard(D_\beta)\} \times \mcard(\alpha)\leq \aleph_{\tau -1} \times \aleph_{\tau -1}\\
    &=\aleph_{\tau -1}.
\end{align*}
\end{proof}
\noindent So for any ordinal $\alpha$ such that $\alpha < \omega_\tau$, we have that Card($D_\alpha) \leq \aleph_{\tau -1}$.
\end{proof}
\end{lemma} 
By Lemma \ref{cardD}, Theorem \ref{thrm:tdhdbase}, and the fact that $D(S_\alpha)=\alpha$ from Theorem \ref{thrm:Salpha}, we have the following result.

\begin{theorem}
For each ordinal $\alpha < \Omega$, there exists a metric space, $X_\alpha$, such that $\alpha \leq$ $\tDHD(X_\alpha) \leq \omega_\tau$, where $\tau$ is the least such ordinal which satisfies $\alpha < \omega_\tau$.
\end{theorem}
The previous theorem can be stated in a weaker manner to see that there exist a metrizable space with $D$-variant transfinite Hausdorff dimension between any two inital ordinals.
\begin{theorem}
    For any two initial ordinals, $\alpha,\beta\in\card\subset\ord$ with $\alpha<\beta$, there exists a metrizable space, $X_{\alpha,\beta}$ such that 
    \[
        \alpha\leq \tDHD(X_{\alpha,\beta}) \leq \beta
    \]
\end{theorem}

One easily checks that Lemmas \ref{metric3} and \ref{lemma:calpha} can be extended to ordinals greater than or equal to $\omega_1$ for a collection of metrizable spaces (as opposed to compact metrizable). Also, it is clear that the metric described in Definition \ref{def:salpha} for $\alpha \geq \omega_1$ is balanced. Then by these extensions and arguments given in Definition \ref{calpha} and Theorem \ref{thrm:Salpha} we have the following result. Namely one may take $X_{\alpha} = C_{\alpha}$ in the following. 
\begin{theorem}
For each $\alpha \in \ord$, there exists a metrizable space, $X_\alpha$, such that $D(X_\alpha)=0$ and $\alpha \leq \tDHD(X_\alpha) \leq \om_{\tau}$ where $\tau\in\ord$ is least such that $\alpha<\om_{\tau}$.
\end{theorem}

\section*{Acknowledgments}
The authors thank Mariusz Urba\'nski for his feedback and guidance in the development of this project.

\bibliographystyle{plain}
\bibliography{DD}
 
\end{document}